\documentclass[11pt,british]{amsart}
\pdfoutput=1

\makeatletter
\@namedef{subjclassname@2020}{
  \textup{2020} Mathematics Subject Classification}
\makeatother

\usepackage[T1]{fontenc}
\usepackage[utf8]{inputenc}
\usepackage{babel}
\usepackage{csquotes}
\usepackage{microtype}

\calclayout

\usepackage{amssymb}
\usepackage{mathtools}
\usepackage[margin=0pt]{subfig}

\newcommand{\N}{\mathbb{N}}
\newcommand{\asd}{\dim_{\mathrm{A}}}
\newcommand{\Rd}{\mathbb{R}^d}
\newcommand{\ubd}{\overline{\dim}_{\mathrm{B}}}
\newcommand{\lbd}{\underline{\dim}_{\mathrm{B}}}
\newcommand{\uid}{\overline{\dim}_{\theta}}
\newcommand{\lid}{\underline{\dim}_{\theta}}
\newcommand{\eps}{\varepsilon}

\theoremstyle{plain}
\newtheorem{theorem}{Theorem}[section]
\newtheorem{lemma}[theorem]{Lemma}
\newtheorem{corollary}[theorem]{Corollary}
\newtheorem{proposition}[theorem]{Proposition}

\numberwithin{equation}{section}

\usepackage[backend=biber,mincrossrefs=9,doi=false,
url=false,isbn=false,
date=year,
giveninits=true,maxnames=99,maxalphanames=99,
sortcites,
hyperref=true,backref=true]{biblatex}
\renewbibmacro{in:}{
  \ifentrytype{article}{}{\printtext{\bibstring{in}\intitlepunct}}}
\DeclareFieldFormat[article,periodical]{volume}{\mkbibbold{#1}}
\DeclareFieldFormat[article,periodical,unpublished]{citetitle}{#1}
\DeclareFieldFormat[article,periodical,unpublished]{title}{#1}
\DeclareFieldFormat{pages}{#1}
\DeclareFieldFormat[inproceedings]{title}{#1\isdot}
\DeclareFieldFormat[misc]{title}{#1\isdot}
\DeclareFieldFormat[misc,article]{note}{#1\addcomma}

\AtEveryBibitem{
  \clearlist{language}
}
\AtEveryBibitem{
  \clearfield{number}}
\title{Dimensions of popcorn-like pyramid sets}

\author{Amlan Banaji}

\author{Haipeng Chen}

\begin{document}

\begin{abstract}
This article concerns the dimension theory of the graphs of a family of functions which include the well-known `popcorn function' and its pyramid-like higher-dimensional analogues. We calculate the box and Assouad dimensions of these graphs, as well as the intermediate dimensions, which are a family of dimensions interpolating between Hausdorff and box dimension. As tools in the proofs, we use the Chung--Erd\H{o}s inequality from probability theory, higher-dimensional Duffin--Schaeffer type estimates from Diophantine approximation, and a bound for Euler's totient function. As applications we obtain bounds on the box dimension of fractional Brownian images of the graphs, and on the H\"older distortion between different graphs. 
\end{abstract}



\maketitle

\section{Introduction and results}

The \emph{popcorn function}, also known as \emph{Thomae's function}, is an important example in real analysis. It has many interesting properties, such as being Riemann integrable despite not being continuous on any open interval. In fact, it is discontinuous at the rationals but continuous at the irrationals. There are several intriguing connections between the popcorn function and different areas of mathematics~\cite{Gorodetski2022,Reeder2021,Copar2020} and computer science~\cite{Shu2017}. In this article, as well as working with the popcorn function itself, we will consider the following higher-dimensional generalisations of it. 
Throughout the paper, $d$ will denote an integer with $d \geq 2$, and $0 < t < \infty$. Then the popcorn pyramid function $f_{t,d} \colon [0,1]^{d-1} \to \mathbb{R}$ is defined by 
\begin{equation}
f_{t,d}(x) = 
\begin{cases}
q^{-t} & \text{if } x=(\frac{p_1}{q},\ldots,\frac{p_{d-1}}{q}), q \in \N, p_i \in \{1,\ldots,q-1\}, \gcd{(p_i,q)} = 1 \, \forall i, \\
0 & \text{otherwise}.
\end{cases}
\end{equation}
The popcorn function itself is $f_{1,2}$. Note that the function is 0 unless all numbers in the product have the same denominator, for example in the $d=3$ case ${f_{t,3}(1/2,1/3)=0}$. The graphs of the functions are denoted by 
\[
G_{t,d} \coloneqq \left\{ (x,f_{t,d}(x)) : x\in [0,1]^{d-1} \right\}. 
\]
Two of the graphs in the $d=2$ case are shown in Figure~\ref{f:sets}; the graph on the left is that of the popcorn function. 
\begin{figure}[t]
\subfloat[The popcorn graph $G_{1,2}$]{\includegraphics[width=.4\textwidth]{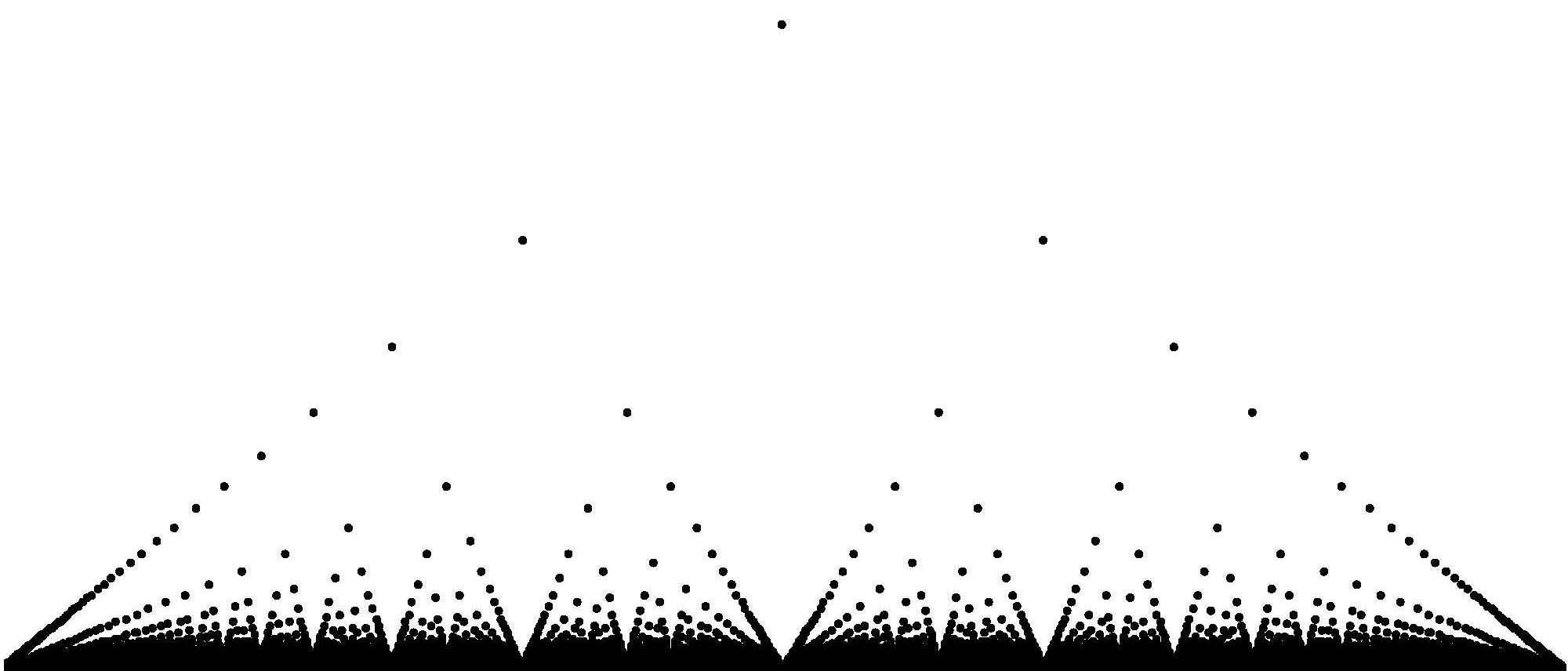}}
\qquad
\subfloat[$G_{0.3,2}$]{\includegraphics[width=.4\textwidth]{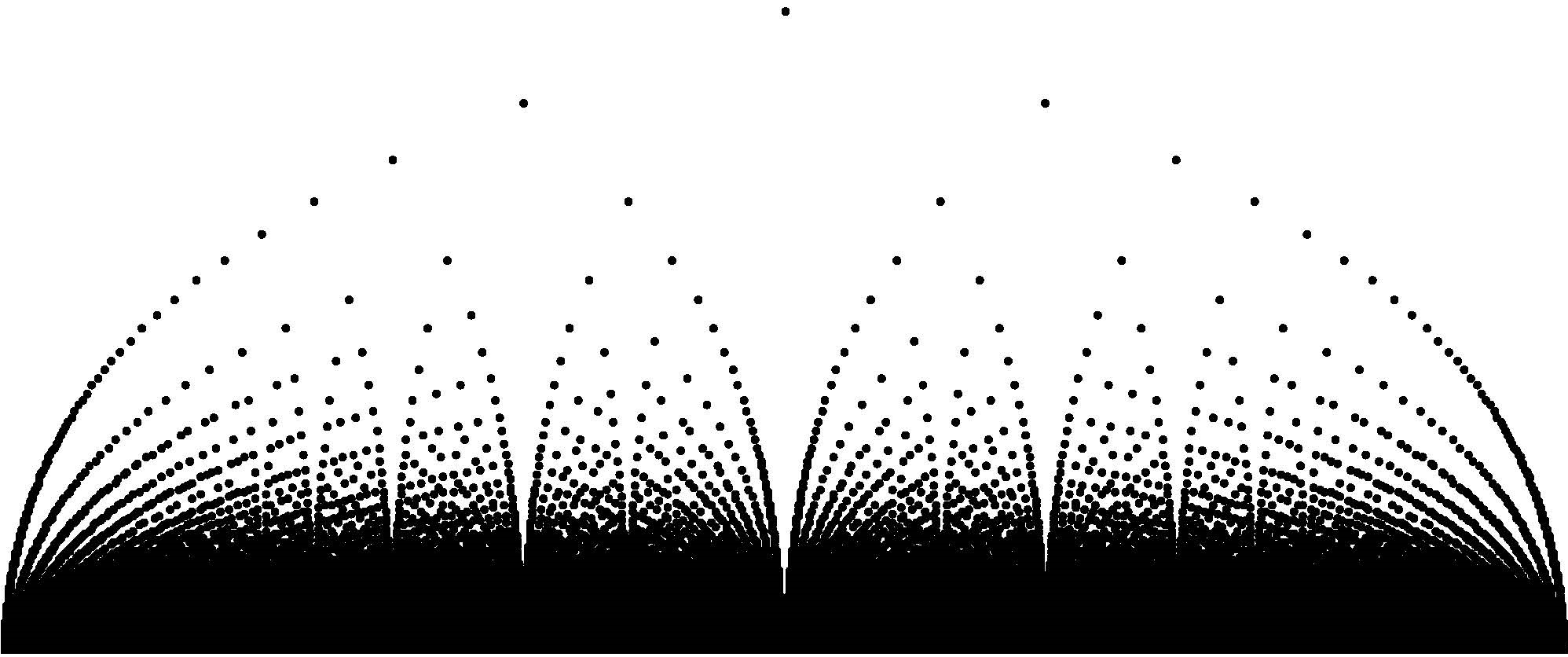}}
\caption{Two popcorn-like graphs}\label{f:sets}
\end{figure}
For completeness, we also include the full set in our analysis, since we will see that the corresponding sets have the same dimensions: 
\[
F_{t,d} \coloneqq \left\{ \, \left( \frac{p_1}{q},\ldots, \frac{p_{d-1}}{q}, \left(\frac{1}{q}\right)^t \right) : q \in \N, p_i \in \{1,\ldots, q-1\} \, \forall i \, \right\} \cup ([0,1]^{d-1} \times \{0\} ).
\]
Then $G_{t,d} \subset F_{t,d} \subset [0,1]^d$; for example, the convex hull of $G_{1,3}$ (or $F_{1,3}$) is a square-based pyramid in $\mathbb{R}^3$, and $(1/2,1/2,1/4) \in  F_{1,3} \setminus G_{1,3}$. 
The sets $G_{t,d}$ and $F_{t,d}$ have an interesting fractal structure, and it is natural to consider different notions of dimension of these sets. This was done in the case $d=2$ in~\cite{Chen2021popcorn,Chen2022tpopcorn}. For an overview of fractal geometry and dimension theory we refer the reader to~\cite{Falconer2014,Fraser2020book}. We will give the formal definitions of the notions of dimension in Section~\ref{s:preliminaries}, prove lower bounds for box dimension in Section~\ref{s:proof1}, and complete the proof of Theorems~\ref{t:assouad} and~\ref{t:int} in Section~\ref{s:proof2}.

Three of the most well-studied notions of fractal dimension are Hausdorff, box and Assouad dimension, which are always ordered 
\begin{equation}\label{e:threedims}
 \dim_{\mathrm H} F \leq \dim_{\mathrm B} F \leq \asd F. \end{equation}
The Hausdorff dimension of $F_{t,d}$ (or $G_{t,d}$) is trivially~$d-1$, since Hausdorff dimension is countably stable. 
The Assouad dimension describes the fine-scale structure of the `thickest' part of the set $F$; for an overview of the Assouad dimension in the context of fractal geometry, we refer the reader to~\cite{Fraser2020book}. We prove that the Assouad dimension of the popcorn-like pyramids are as follows. 
\begin{theorem}\label{t:assouad}
We have 
\[ 
\dim_{\mathrm A} G_{t,d} = \dim_{\mathrm A} F_{t,d} = \begin{cases} d, & 0 < t < d/(d-1), \\
d-1, & t \geq d/(d-1). \end{cases} 
\]
\end{theorem}
It follows from Theorem~\ref{t:assouad} and~\eqref{e:threedims} that if $t\geq d/(d-1)$ then all three notions of dimension in~\eqref{e:threedims} are equal to $d-1$, so in all other results in this paper we assume that $t < d/(d-1)$. 
Box dimension describes the scaling behaviour of the smallest number of balls of a given size needed to cover the set. 
\begin{theorem}\label{t:box}
If $0 < t < d/(d-1)$ then 
\[ \dim_{\mathrm B} G_{t,d} = \dim_{\mathrm B} F_{t,d} = \frac{d^2}{d + t}. \]
\end{theorem}
\begin{proof}
This follows by setting $\theta=1$ in the more general result Theorem~\ref{t:int} below. 
\end{proof}
Note that $\dim_{\mathrm B} G_{t,d}$ is continuous in $t$ but $\asd G_{t,d}$ is not, since Assouad dimension depends sensitively on the local scaling behaviour of the set. 
Although our proof of Theorem~\ref{t:box} is to observe that it follows from Theorem~\ref{t:int}, we do not call it a corollary because perhaps the most challenging part of the proof of Theorem~\ref{t:int} is to establish lower bounds for box dimension. We do this directly in Section~\ref{s:proof1} using a similar strategy to the $d=2$ case in~\cite{Chen2021popcorn,Chen2022tpopcorn}. Indeed, to bound the covering number, we make use of the Chung--Erd\H{o}s inequality from probability theory. One part of the resulting expression can be bounded using the estimate $\phi(n) \gtrsim n/(\log \log n)$ for Euler's totient function $\phi(n)$ (defined in~\eqref{e:defineeuler}). 
Another part is bounded using methods related to Diophantine approximation. Indeed, in the paper~\cite{Duffin1941schaeffer} in which Duffin and Schaeffer formulated their famous conjecture (which was recently proved by Koukoulopoulos and Maynard~\cite{Koukoulopoulos2020}), they also established bounds for intersections of sets of numbers which can be approximated by rationals in a certain way. For our results in the $d>2$ case, we use bounds on higher-dimensional Diophantine approximation sets which can be deduced from the bounds in~\cite{Duffin1941schaeffer}. These bounds are also noted by Pollington and Vaughan in~\cite{Pollington1990duffinschaeffer}, where they prove the higher-dimensional version of the Duffin--Schaeffer conjecture (which is substantially less challenging than the proof of the one-dimensional version).

A topic which has generated significant interest in recent years is dimension interpolation. The idea is to consider two different notions of dimension and find a geometrically meaningful family of dimensions that lies between them. For a survey of this topic, we refer the reader to~\cite{Fraser2021interpolating}. 
One example of dimension interpolation is the Assouad spectrum, which lies between the box and Assouad dimensions. The Assouad spectrum of the popcorn graph and the sets $G_{t,2}$ and $F_{t,2}$ has been computed in~\cite{Chen2021popcorn,Chen2022tpopcorn}. A natural question would therefore be to calculate the Assouad spectrum of the sets $G_{t,d}$ and $F_{t,d}$ as a function of $d$, $t$ and $\theta$. 
The spectrum of dimensions which we will consider in this paper, however, is the \emph{intermediate dimensions}, introduced in~\cite{Falconer2020firstintermediate}. 
These depend on a parameter $\theta \in [0,1]$ and interpolate between Hausdorff dimension ($\theta = 0$) and box dimension ($\theta = 1$). The motivation for these dimensions comes from the observation that the box dimension is defined by covering the set with balls of a fixed size, whereas for the Hausdorff dimension the covering sets can have very different sizes. 
The intermediate dimensions are defined by covering the set in question with sets whose diameter lies in intervals of the form $[\delta^{1/\theta},\delta]$ for a small number $\delta > 0$.
 A precise characterisation of the attainable forms of intermediate dimensions of sets has been obtained in~\cite{Banaji2021moran}; for a survey of the intermediate dimensions, we refer the reader to~\cite{Falconer2021intdimsurvey}. 
The following result extends Theorem~\ref{t:box}. 
\begin{theorem}\label{t:int}
If $0 < t < d/(d-1)$ then 
\[ 
\dim_{\theta} G_{t,d} = \dim_{\theta} F_{t,d} =\begin{cases}
d-1, & 0 \leq \theta \leq \frac{(d-1)t}{d}, \\
\frac{d^2 \theta}{d\theta + t}, & \frac{(d-1)t}{d} < \theta \leq 1.
\end{cases}
\] 
\end{theorem}
We see that after the phase transition at $\theta = (d-1)t/d$, the graph of the intermediate dimensions is analytic and strictly concave. 
One simple way to complete the proof is to make use of Banaji's general bounds~\eqref{e:generalbound} and~\eqref{e:generallower}, but we also give a direct proof. 
As a special case of Theorem~\ref{t:int}, we obtain a formula for the intermediate dimensions of the graph of the popcorn function. 
\begin{corollary}
\[ \dim_{\theta} G_{1,2} = \dim_{\theta} F_{1,2} = \begin{cases}
1 & 0 \leq \theta \leq 1/2 \\
\frac{4 \theta}{2\theta + 1} & 1/2 < \theta \leq 1
\end{cases} 
\] 
\end{corollary}
\begin{proof}
This is immediate from Theorem~\ref{t:int}.
\end{proof}

It was shown in~\cite{Falconer2020firstintermediate} that the maps $\theta \mapsto \uid F$ and $\theta \mapsto \lid F$ are always continuous for $\theta \in (0,1]$, but for some sets $F$ they are not continuous at $\theta = 0$. Continuity at $\theta = 0$ has been shown in~\cite{Burrell2022brownian,Burrell2021projections} to have powerful consequences. In particular, in the following corollary of Theorem~\ref{t:int}, we apply results of Burrell~\cite{Burrell2022brownian} to give bounds for the box dimension of images of the sets $G_{t,d}$ and $F_{t,d}$ under fractional Brownian motion (which is a stochastic process defined and studied in Kahane's classical text~\cite{Kahane1985}). 
Falconer~\cite{Falconer2021} has explicitly computed the intermediate dimensions of fractional Brownian images of certain sequence sets. 
\begin{corollary}
Fix $d \in \mathbb{N}$ with $d \geq 2$, $0 < t < d/(d-1)$ and $\alpha > (d-1)/d$. 
If $B_\alpha \colon \Rd \to \Rd$ is index-$\alpha$ fractional Brownian motion then 
\begin{align*}
\dim_{\mathrm H} B_\alpha(G_{t,d}) &= \dim_{\mathrm H} B_\alpha(F_{t,d}) = (d-1)/\alpha \\*
\ubd B_\alpha(G_{t,d}) &\leq \ubd B_\alpha(F_{t,d}) < d
\end{align*} 
hold almost surely. 
\end{corollary}
\begin{proof}
The value of the Hausdorff dimension of the fractional Brownian image is a direct consequence of Kahane's general results~\cite[Chapter~18]{Kahane1985}, since $G_{t,d}$ and $F_{t,d}$ are Borel. The box dimension result follows from Burrell's~\cite[Corollary~3.7]{Burrell2022brownian}, since the intermediate dimensions of $G_{t,d}$ and $F_{t,d}$ are continuous at $\theta = 0$ by Theorem~\ref{t:int}. 
\end{proof}
It is interesting to note that the condition $\alpha > (d-1)/d$ does not depend on $t$, even though the box dimension of the sets $G_{t,d}$ and $F_{t,d}$ does depend on $t$. 

Recall that for $\alpha \in (0,1]$ we say that a map is \emph{$\alpha$--H{\"o}lder} if that there exists $c>0$ such that $||f(x)-f(y)|| \leq c||x-y||^\alpha$ for all $x,y \in F$. 
Dimension theory can play a role in determining the H\"older distortion between sets; for a discussion of this topic, we refer the reader to~\cite[Section~17.10]{Fraser2020book}. The intermediate dimensions can often give better information than either the Hausdorff or box dimensions, for example in the case of Bedford--McMullen carpets (see~\cite[Example~2.12]{Banaji2021bedford}) and continued fraction sets (see~\cite[Example~4.5]{Banaji2021}). In fact, we show in Corollary~\ref{c:holder} that the same is true for the sets $G_{t,d}$ and $F_{t,d}$. 
To do so, we use the following simple estimate noted by Falconer~\cite[Section~2.1 5.]{Falconer2021intdimsurvey}: if $f \colon F \to \Rd$ is $\alpha$-H{\"o}lder, then $\uid f(F) \leq \alpha^{-1} \uid F$ and $\lid f(F) \leq \alpha^{-1} \lid F$. For further H{\"o}lder distortion estimates for the intermediate dimensions we refer the reader to~\cite[Theorem~3.1]{Burrell2022brownian} and~\cite[Section~4]{Banaji2020}. The bound achieved by different values of $\theta$ for a certain choice of parameters is shown in Figure~\ref{f:holder}. 

\begin{corollary}\label{c:holder}
Suppose $0 < t_1 < t_2 \leq d/(d - 1)$. Then if $f \colon G_{t_2,d} \to \mathbb{R}^{d}$ satisfies $f(G_{t_2,d}) \supseteq G_{t_1,d}$ and is $\alpha$-H\"older, then 
\[ \alpha \leq \frac{(d - 1)t_2 + t_1}{d t_2}. \]
The same holds if $G_{t_1,d}$ is replaced by $F_{t_1,d}$ or $G_{t_2,d}$ is replaced by $F_{t_2,d}$. 
\end{corollary}
\begin{proof}
If $\theta = (d-1)t_2/d$ then 
\[ \alpha \leq \frac{\uid G_{t_2,d}}{\uid f(G_{t_2,d})} \leq \frac{\dim_{\theta} G_{t_2,d}}{\dim_{\theta} G_{t_1,d}} = \frac{(d - 1)t_2 + t_1}{d t_2}. \qedhere \]
\end{proof}
It is straightforward to see that the value of $\theta$ which gives the best bound for $\alpha$ in the proof of Corollary~\ref{c:holder} is indeed $\theta = (d-1)t_2/d$ (the largest value $\theta$ for which $\dim_\theta G_{t_2,d} = d-1$). It may be of interest to determine whether the bounds in Corollary~\ref{c:holder} are sharp, but we will not pursue this. 
It follows from Corollary~\ref{c:holder} (or directly from Theorem~\ref{t:box}) that if $0 < t_1 < t_2 \leq d/(d - 1)$ then $G_{t_1,d}$ and $G_{t_2,d}$ are not bi-Lipschitz equivalent. 

\begin{figure}[t]
\centering
\includegraphics[width=.9\textwidth]
        {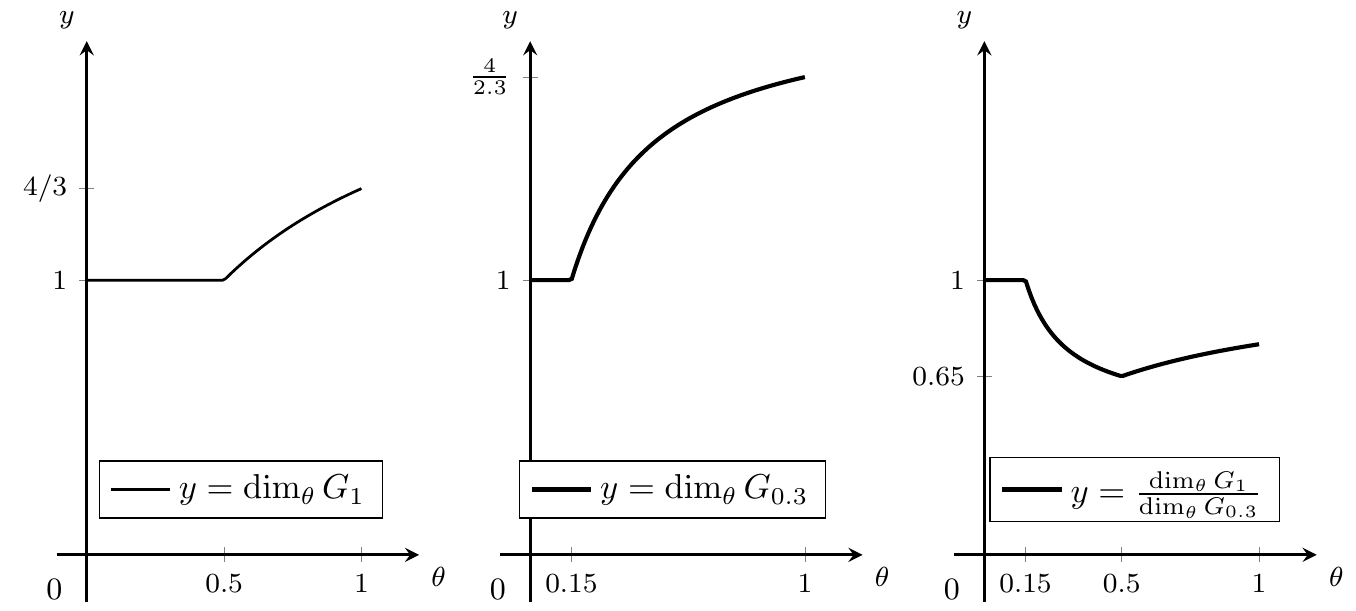}
        \caption{\label{f:holder}
        Graph of the intermediate dimensions of the popcorn sets from Figure~\ref{f:sets}, and their ratio (which gives upper bounds on the possible H\"older exponents of surjective maps from $G_{1,2}$ to $G_{0.3,2}$). 
 }
\end{figure}

\section{Preliminaries}\label{s:preliminaries}
We denote the cardinality of a set $S$ by $\# S$. Throughout, we write $a \lesssim b$ to mean $a \leq c b$ for some constant $c$ which may depend on $d$ or $t$ but is independent of other parameters (such as $\delta$) unless stated otherwise. For each $x > 0$, we denote 
\[ \lfloor x \rfloor \coloneqq \max \{ \, n \in \mathbb{N}^+ : n \leq x \, \}.\] 
 Let $\mathcal{L}_{n}$ be the Lebesgue measure on $\mathbb{R}^{n}$. All sets we consider will be non-empty, bounded subsets of Euclidean space. Given such a set $F$, we denote its diameter by $|F| \coloneqq \sup \{ \, ||x-y|| : x,y \in F \, \}$, where $|| \cdot ||$ is Euclidean distance. 
Let $N_\delta (F)$ denote the minimal number of closed cubes with side length $\delta$ needed to cover $F$. 
The \emph{upper} and \emph{lower box dimensions} are defined by 
\begin{equation}\label{e:boxdimdef}
 \ubd F \coloneqq \limsup_{\delta \to 0^+} \frac{\log N_\delta(F)}{-\log \delta}; \qquad \lbd F \coloneqq \liminf_{\delta \to 0^+} \frac{\log N_\delta(F)}{-\log \delta}.
 \end{equation}
If these coincide, then we simply refer to the \emph{box dimension} of $F$, denoted $\dim_{\mathrm B} F$. 
We can define the \emph{Hausdorff dimension} without using Hausdorff measure by 
\begin{align*}
\dim_\mathrm{H} F = \inf \{ \, s \geq 0 : &\mbox{ for all } \epsilon >0 \mbox{ there exists a finite or countable cover } \\*
&\{U_1,U_2,\ldots\} \mbox{ of } F \mbox{ such that } \sum_i |U_i|^s \leq \eps \,\}.
\end{align*}
This motivates the definition of the intermediate dimensions. For $\theta \in (0,1]$, the \emph{upper $\theta$-intermediate dimension} of $F$ is
\begin{align*} \uid F = \inf \{ \, s \geq 0 : &\mbox{ for all } \epsilon >0 \mbox{ there exists } \delta_0 \in (0,1] \mbox{ such that for all } \\*
&\delta \in (0,\delta_0) \mbox{ there exists a cover } \{U_1,U_2,\ldots\} \mbox{ of } F \\*
&\mbox{ such that } \delta^{1/\theta} \leq |U_i| \leq \delta \mbox{ for all } i, \mbox{ and } \sum_i |U_i|^s \leq \eps \, \}.
\end{align*}
Similarly the \emph{lower $\theta$-intermediate dimension} of $F$ is 
\begin{align*} \lid F = \inf \{ \, s \geq 0 : &\mbox{ for all } \epsilon >0 \mbox{ and } \delta_0 \in (0,1] \mbox{ there exists } \delta \in (0,\delta_0) \\* &\mbox{ and a cover } \{U_1,U_2,\ldots\} \mbox{ of } F \mbox{ such that } \delta^{1/\theta} \leq |U_i| \leq \delta \\* 
&\mbox{ for all } i, \mbox{ and } \sum_i |U_i|^s \leq \eps \,\}.
\end{align*}
If these coincide, then we refer to the \emph{intermediate dimension} of $F$, denoted $\dim_{\theta} F$. By definition, $\overline{\dim}_{1} = \ubd$ and $\underline{\dim}_{1} = \lbd$, and $\overline{\dim}_{0} = \underline{\dim}_{0} \coloneqq \dim_\mathrm{H}$. 

 Finally, the \emph{Assouad dimension} of $F \subseteq \Rd$ is defined by 
\begin{multline*}
      \dim_\mathrm{A} F = \inf\left\{
      \alpha : \mbox{ there exists }C>0\mbox{ such that for all } x \in F \mbox{ and } 
    \right. \\ \qquad \left. 0<r<R, \mbox{ we have } N_r(B(x,R)\cap F) \leq C(R/r)^\alpha\right\}. 
    \end{multline*}
    The following inequalities are always satisfied 
\begin{gather*}
 0 \leq \dim_\mathrm{H} F \leq \lid F \leq \uid F \leq \ubd F \leq \asd F \leq d \\
  \lid F \leq \lbd F \leq \ubd F. 
  \end{gather*}
It will be useful to recall the following general bounds for subsets of $\Rd$, proved in~\cite[Propositions~3.8 and~3.10]{Banaji2020}. We will see that these bounds are sharp for large $\theta$ for the families $G_{t,d}$ and $F_{t,d}$. 
If $0 < \theta \leq \phi \leq 1$, then 
\begin{equation}\label{e:generalbound}
\overline{\dim}_\theta F \leq \overline{\dim}_\phi F \leq \overline{\dim}_\theta F + \frac{\overline{\dim}_\theta F (\asd F - \overline{\dim}_\theta F )}{(\phi-\theta) \overline{\dim}_\theta F  + \theta \asd F}(\phi-\theta) . 
\end{equation}
Moreover, there is a lower bound for the intermediate dimensions in terms of the box and Assouad dimensions: 
\begin{equation}\label{e:generallower}
\overline{\dim}_\theta F \geq \frac{\theta\asd F \ubd F}{\asd F - (1-\theta)\ubd F}.
\end{equation}
Both bounds hold with $\overline{\dim}$ replaced by $\underline{\dim}$ throughout. It is clear that if $\asd F$ is replaced by $d$ throughout then the bounds will hold for all subsets of $\Rd$. 

In Section~\ref{s:proof1} we will use the Chung--Erd\H{o}s inequality (see \cite[equation~(4)]{Chung1952erdos} and \cite[equation~(6.10)]{Petrov1995}) which states that if $A_1,\dotsc,A_m$ are events in a probability space $(X, \mathcal{X}, \mu)$ and least one of them is positive, then  
\begin{equation}\label{e:chung}
 \mu\left(\bigcup_{i=1}^m A_i \right) \geq \frac{\left(\sum_{j=1}^m \mu(A_j)\right)^2}{\sum_{k=1}^m \sum_{l=1}^m \mu(A_k \cap A_l)}.
 \end{equation} 
We will also use the \emph{Euler totient function} from number theory, defined by 
\begin{equation}\label{e:defineeuler}
 \phi(n) \coloneqq \# \{ \, m : \gcd(m,n)=1, 1 \leq m < n \, \}.
 \end{equation}
Specifically we will use the following bound on its growth (see \cite[Theorem~2.9]{Montgomery2006}): for integers $n \geq 2$, 
\begin{equation}\label{e:totientbound}
\phi(n) \gtrsim \frac{n}{\log \log n},
\end{equation}
with the implicit constant independent of $n$. 
This bound suffices for our purposes because logarithmic factors will not affect dimension estimates.

\section{Lower bound for box dimension}\label{s:proof1}
The most challenging argument is the proof of the lower bound for box dimension, which follows a similar structure to the $d=2$ case in~\cite{Chen2021popcorn,Chen2022tpopcorn}, using tools from probability theory and number theory. In light of Theorem~\ref{t:assouad}, in this section we make the standing assumption that $0 < t < d/(d-1)$. We first introduce some notation and make some preliminary claims. 
It follows from the definition of $G_{t,d}$ that
\begin{align*}
G_{t,d} \cap ([0,1]^{d-1} \times (0,1]) = \bigcup_{n = 2}^\infty \Big\{ & \Big(\frac{m_1}{n}, \dots, \frac{m_{d-1}}{n}, \frac{1}{n^t}\Big) : \\*
& \qquad  \gcd{(m_i,n)}=1, 1\leq m_i <n, 1 \leq i \leq d-1 \Big\}.
\end{align*}
Given $0 < t < \frac{d}{d-1}$, $\delta > 0$ and $1 \leq k \leq \lfloor \delta^{-1} \rfloor + 1$, we denote 
\begin{equation}\label{e:definel}
l_{t}(k,\delta) \coloneqq \left\lfloor \left(\frac{1}{k\delta} \right)^{\frac{1}{t}} \right\rfloor.
\end{equation}
Given sufficiently small $0 < \delta < 1$, for each $k \geq 1$, we define the following sets of input values which give rise to \emph{horizontal layers of $G_{t,d}$} of height $\delta$: 
\begin{align*}
L_{t,d}(\delta,k) \coloneqq \bigcup_{n = l_t(k+1,\delta) + 1}^{l_t(k,\delta)} \Big\{ \Big(\frac{m_1}{n}, \dots, \frac{m_{d-1}}{n}\Big) : \gcd{(m_i,n)}=1, &1\leq m_i <n, \\*
&1 \leq i \leq d-1 \Big\}. 
\end{align*}
The following lemma gives some explanation to the choice of some parameters in the lemmas below.  
\begin{lemma}\label{epsilon}
For all $0 < \varepsilon < \frac{1}{16} \cdot \min \left\{  \left( \frac{dt}{t+d} - \frac{t}{t+1}\right), \left(1-\frac{dt}{d+t}\right)\right\} $, for all sufficiently small $\delta > 0$ and for all $\lfloor \delta^{\frac{dt}{t+d}-1-\varepsilon} \rfloor \leq k \leq \lfloor \delta^{\frac{t}{t+1}-1+\varepsilon} \rfloor$, 
\begin{enumerate}
\item There exists at least one $n \in \mathbb{N}$ such that $k\delta \leq \frac{1}{n^t} < (k+1)\delta$.
\item If $n$ is such that $k\delta \leq \frac{1}{n^t} < (k+1)\delta$ then $\frac{1}{n} > \delta$.
\end{enumerate}
\end{lemma}

\begin{proof}
~~

\begin{enumerate}
\item 
Suppose $\lfloor \delta^{\frac{dt}{t+d}-1-\varepsilon} \rfloor \leq k \leq \lfloor \delta^{\frac{t}{t+1}-1+\varepsilon} \rfloor$, and let $n \geq 2$ be such that $n^{-t} < k\delta$. Then since $k\delta < \delta^{\frac{t}{t+1}+ \frac{\varepsilon}{2}}$, by the mean value theorem we obtain 
\[
\frac{1}{(n-1)^t} - \frac{1}{n^t} \lesssim \frac{1}{n^{t+1}} < \delta.
\]
\item If $k\delta \leq \frac{1}{n^t} < (k+1)\delta$ then
\[
\frac{1}{n} \geq (k\delta)^{\frac{1}{t}} > \delta^{\frac{d}{t+d}} > \delta. \qedhere
\]
\end{enumerate}
\end{proof}

We now introduce the \emph{higher-dimensional Duffin--Schaeffer type estimate}. For each $F \subseteq \mathbb{R}^1$, we denote $(\mathcal{L}_1(F))^{d-1}$ by $\mathcal{L}_1^{d-1}(F)$. For each real function $\psi \colon \mathbb{N} \mapsto \mathbb{R}^1$, we write
\[
E^{(1)}(\psi(n),n) = \left\{ \left( \frac{m}{n} - \frac{\psi(n)}{n}, \frac{m}{n} + \frac{\psi(n)}{n}\right) : \gcd{(m,n)}=1, 1\leq m <n\right\}.
\]
It follows that $E(\psi(n), n)$ is a finite union of intervals with diameter $\frac{2\psi(n)}{n}$ provided $\psi(n) < 1/2$. For the $d$-dimensional case, we write
\begin{equation}\label{e:defineEsets}
\begin{aligned}
E^{(d-1)}(\psi(n),n) =  \Big\{& \left( \frac{m_1}{n} - \frac{\psi(n)}{n}, \frac{m_1}{n} + \frac{\psi(n)}{n}\right) \times \dots \\*
&\qquad \qquad \times  \left(   \frac{m_{d-1}}{n} - \frac{\psi(n)}{n}, \frac{m_{d-1}}{n} + \frac{\psi(n)}{n}\right)  : \\*
& \qquad \qquad \gcd{(m_i,n)}=1, 1\leq m_i <n, 1 \leq i \leq d-1 \Big\}.
\end{aligned}
\end{equation}

For the rest of this section, we consider $\psi(n) = n\delta$, and simplify $E^{(1)}(\psi(n),n)$ and $E^{(d-1)}(\psi(n),n)$ to $E^{(1)}(\delta,n)$ and $E^{(d-1)}(\delta,n)$ respectively. 
\begin{proposition}[Higher dimensional Duffin--Schaeffer type estimate]\label{higher_DS}
There exists a constant $c$ (depending only on $d$) such that for all distinct integers $q,k \geq 2$,  
\[
\mathcal{L}_{d-1}(E^{(d-1)}(\delta,q) \cap E^{(d-1)}(\delta,k)) \leq c \cdot (q \cdot k \cdot \delta^2)^{d-1}.
\]
\end{proposition}

\begin{proof}
By the one-dimensional Duffin--Schaeffer estimate \cite[Lemma~2]{Duffin1941schaeffer}, 
\[
\mathcal{L}_1(E^{(1)}(\delta,q) \cap E^{(1)}(\delta,k)) \leq 4 \cdot q \cdot k \cdot \delta^2.
\] 
The desired estimate now follows from the fact that by the definition of Lebesgue measure on $\mathbb{R}^{d-1}$ and $E^{(d-1)}(\delta,q)$,
\[
\mathcal{L}_{d-1}(E^{(d-1)}(\delta,q) \cap E^{(d-1)}(\delta,k))  = \mathcal{L}_1^{d-1}(E^{(1)}(\delta,q) \cap E^{(1)}(\delta,k))
\]
(see also \cite{Pollington1990duffinschaeffer}). 
\end{proof}
The following lemma estimates $N_{\delta}(L_{t,d}(\delta,k))$.
\begin{lemma}\label{l:coveringlebesgue}
Fix $\varepsilon > 0$ as in Lemma~\ref{epsilon}. For all sufficiently small $\delta > 0$ and all $\lfloor \delta^{\frac{dt}{t+d}-1-\varepsilon} \rfloor \leq k \leq \lfloor \delta^{\frac{t}{t+1}-1+\varepsilon} \rfloor$, 
\[
\delta^{d-1} N_{\delta}(L_{t,d}(\delta,k)) \approx \mathcal{L}_{d-1} \left( \bigcup_{n = l_t(k+1,\delta) + 1}^{l_t(k,\delta)} E^{(d-1)}(\delta,n) \right).
\]
\end{lemma}

\begin{proof}
Since $\bigcup_{n = l_t(k+1,\delta) + 1}^{l_t(k,\delta)} E^{(d-1)}(\delta,n)$ is a finite union of cubes in $\mathbb{R}^{d-1}$ with side length $2\delta$, the result follows by a straightforward geometric argument. 
\end{proof}

For the estimate of the lower bound of $N_{\delta}(L_{t,d}(\delta,k))$,
it follows from the Chung--Erd\H{o}s inequality~\eqref{e:chung} that
\begin{align}\label{e:chungerdosspecific}
\begin{split}
\mathcal{L}_{d-1} &\left(\bigcup\limits_{n = l_t(k+1,\delta) + 1}^{l_t(k,\delta)} E^{(d-1)}(\delta,n) \right) \\*
&\geq \frac{\left( \sum\limits_{n = l_t(k+1,\delta) + 1}^{l_t(k,\delta)} \mathcal{L}_{d-1} \left( E^{(d-1)}(\delta,n) \right) \right)^2}{\sum\limits_{n = l_t(k+1,\delta) + 1}^{l_t(k,\delta)} \sum\limits_{m = l_t(k+1,\delta) + 1}^{l_t(k,\delta)} \mathcal{L}_{d-1} \left(E^{(d-1)}(\delta,n) \cap E^{(d-1)}(\delta,m) \right) }.
\end{split}
\end{align}
The following lemmas now give estimates for the nominator and the denominator respectively of the right-hand side of the expression above.

\begin{lemma}\label{nominator}
Fix $\varepsilon > 0$ as in Lemma~\ref{epsilon}. For all sufficiently small $\delta > 0$ and all $\lfloor \delta^{\frac{dt}{t+d}-1-\varepsilon} \rfloor \leq k \leq \lfloor \delta^{\frac{t}{t+1}-1+\varepsilon} \rfloor$, 
\[
\delta^{d-1} \cdot \frac{1}{k^{\frac{d}{t}+1}\delta^{\frac{d}{t}}} \cdot \left( \frac{1}{\log \log \delta^{-\frac{d}{t+d}}} \right)^d \lesssim \sum\limits_{n = l_t(k+1,\delta) + 1}^{l_t(k,\delta)} \mathcal{L}_{d-1} \left( E^{(d-1)}(\delta,n) \right) \lesssim  \frac{\delta^{d-1}}{k^{\frac{d}{t}+1}\delta^{\frac{d}{t}}} .
\]
\end{lemma}

\begin{proof}
For the upper bound, 
\begin{align*}
\sum\limits_{n = l_t(k+1,\delta) + 1}^{l_t(k,\delta)} \mathcal{L}^{d-1} \left( E^{(d-1)}(\delta,n) \right) &\lesssim \sum\limits_{n = l_t(k+1,\delta) + 1}^{l_t(k,\delta)} n^{d-1} \delta^{d-1} &\text{by~\eqref{e:defineEsets}}\\
&\lesssim \delta^{d-1} \cdot \left(\left(\frac{1}{k\delta}\right)^{\frac{d}{t}} - \left(\frac{1}{(k+1)\delta}\right)^{\frac{d}{t}} \right) &\text{by \eqref{e:definel}}\\
&\lesssim \delta^{d-1} \cdot \frac{1}{k^{\frac{d}{t}+1}\delta^{\frac{d}{t}}},
\end{align*}
with the last line by the mean value theorem. 
For the lower bound, it follows from Lemma~\ref{epsilon} that the distance between centres of cubes in $E^{(d-1)}(\delta,n)$ is larger than $\delta$. Therefore it follows from the bound~\eqref{e:totientbound} for the Euler totient function that
\begin{equation}\label{e:secondeulerestimate}
\sum\limits_{n = l_t(k+1,\delta) + 1}^{l_t(k,\delta)} \mathcal{L}_{d-1} \left( E_{(d-1)}(\delta,n) \right) \gtrsim \sum\limits_{n = l_t(k+1,\delta) + 1}^{l_t(k,\delta)} \left( n \delta \cdot \frac{1}{\log \log n}\right)^{d-1}.
\end{equation}
Since $\lfloor \delta^{\frac{dt}{t+d}-1-\varepsilon} \rfloor \leq k \leq \lfloor \delta^{\frac{t}{t+1}-1+\varepsilon} \rfloor$, if $n$ is such that $l_t(k+1,\delta) < n \leq l_t(k,\delta)$ then $\delta^{\frac{dt}{d+t}} < \frac{1}{n^t} < \delta^{\frac{t}{1+t}}$ so $\delta^{-\frac{1}{1+t}}\leq n \leq \delta^{-\frac{d}{d+t}}$, and hence 
\begin{equation}\label{e:loglogbound}
\frac{1}{\log \log n} \geq \frac{1}{\log \log \delta^{-\frac{d}{d+t}}}.
\end{equation}
Moreover, by the mean value theorem,
\begin{equation}\label{e:mvtbound}
\sum\limits_{n = l_t(k+1,\delta) + 1}^{l_t(k,\delta)} n^{d-1} \gtrsim  \left(\frac{1}{k\delta}\right)^{\frac{d}{t}} - \left(\frac{1}{(k+1)\delta}\right)^{\frac{d}{t}} \gtrsim  \frac{1}{k^{\frac{d}{t}+1}\delta^{\frac{d}{t}}}.
\end{equation}
Combining~\eqref{e:secondeulerestimate},~\eqref{e:loglogbound} and~\eqref{e:mvtbound} gives the desired result.
\end{proof}

\begin{lemma}\label{denominator}
Fix $\varepsilon > 0$ as in Lemma~\ref{epsilon}. For all sufficiently small $\delta > 0$ and all $\lfloor \delta^{\frac{dt}{t+d}-1-\varepsilon} \rfloor \leq k \leq \lfloor \delta^{\frac{t}{t+1}-1+\varepsilon} \rfloor$, 
\begin{equation*}
\sum\limits_{n = l_t(k+1,\delta) + 1}^{l_t(k,\delta)}  \sum\limits_{m = l_t(k+1,\delta) + 1}^{l_t(k,\delta)}  \mathcal{L}_{d-1} \left(E^{(d-1)}(\delta,n) \cap E^{(d-1)}(\delta,m) \right) \lesssim \delta^{d-1} \cdot \frac{1}{k^{\frac{d}{t}+1}\delta^{\frac{d}{t}}}.
\end{equation*}
\end{lemma}

\begin{proof}
First note that 
\begin{align*}
 \sum\limits_{n = l_t(k+1,\delta) + 1}^{l_t(k,\delta)}  &\sum\limits_{m = l_t(k+1,\delta) + 1}^{l_t(k,\delta)} \mathcal{L}_{d-1} \left(E^{(d-1)}(\delta,n) \cap E^{(d-1)}(\delta,m) \right) \\
 &= \sum_{\substack{n,m = l_t(k+1,\delta) + 1 \\ m \neq n}}^{l_t(k,\delta)} \mathcal{L}_{d-1}  \left(E^{(d-1)}(\delta,n) \cap E^{(d-1)}(\delta,m) \right) \\*
 &\quad+ \sum\limits_{n = l_t(k+1,\delta) + 1}^{l_t(k,\delta)}\mathcal{L}_{d-1} \left( E^{(d-1)}(\delta,n) \right).
\end{align*}
By the upper bound of Lemma~\ref{nominator} and the range of $k$ under consideration 
\[ 
\sum\limits_{n = l_t(k+1,\delta) + 1}^{l_t(k,\delta)}\mathcal{L}_{d-1} \left( E^{(d-1)}(\delta,n) \right) \lesssim  \delta^{d-1} \cdot \frac{1}{k^{\frac{d}{t}+1}\delta^{\frac{d}{t}}} \lesssim 1.
\]
By the higher dimensional Duffin--Schaeffer type estimate (Proposition~\ref{higher_DS}), for all $m \neq n$, we obtain
\[
\mathcal{L}_{d-1}  \left(E^{(d-1)}(\delta,n) \cap E^{(d-1)}(\delta,m) \right) \lesssim (m\cdot n\cdot \delta^2)^{d-1}.
\]
Hence, again using the same calculation as the proof of the upper Lemma~\ref{nominator}, 
\begin{align*}
\sum_{\substack{n,m = l_t(k+1,\delta) + 1 \\ m \neq n}}^{l_t(k,\delta)} \mathcal{L}_{d-1}  \left(E^{(d-1)}(\delta,n) \cap E^{(d-1)}(\delta,m) \right) &\lesssim \left( \sum\limits_{n = l_t(k+1,\delta) + 1}^{l_t(k,\delta)} \delta^{d-1} \cdot n^{d-1} \right)^2 \\
&\lesssim \left( \delta^{d-1} \cdot \frac{1}{k^{\frac{d}{t}+1}\delta^{\frac{d}{t}}} \right)^2 \\
&\lesssim  \delta^{d-1} \cdot \frac{1}{k^{\frac{d}{t}+1}\delta^{\frac{d}{t}}},
\end{align*}
completing the proof of the bound. 
\end{proof}

Bringing together the above bounds, we obtain the following important estimate. 
\begin{lemma}\label{CE-final}
Fix $\varepsilon > 0$ as in Lemma~\ref{epsilon}. For all sufficiently small $\delta > 0$ and all $\lfloor \delta^{\frac{dt}{t+d}-1-\varepsilon} \rfloor \leq k \leq \lfloor \delta^{\frac{t}{t+1}-1+\varepsilon} \rfloor$, 
\[
 N_{\delta}(f_{t,d}(L_{t,d}(\delta,k))) \approx N_{\delta}(L_{t,d}(\delta,k)) \gtrsim \frac{1}{k^{\frac{d}{t}+1}\delta^{\frac{d}{t}}} \cdot \left( \frac{1}{\log \log \delta^{-\frac{d}{t+d}}} \right)^{2d}.
\]
\end{lemma}

\begin{proof}
The first approximate equality is since the height of each strip $f_{t,d}(L_{t,d}(\delta,k))$ is $\approx \delta$. 
By Lemma~\ref{l:coveringlebesgue}, 
\[
N_{\delta}(L_{t,d}(\delta,k)) \approx \frac{1}{\delta^{d-1}} \cdot \mathcal{L}_{d-1} \left(\bigcup\limits_{n = l_t(k+1,\delta) + 1}^{l_t(k,\delta)} E^{(d-1)}(\delta,n) \right).
\]
Since $k > \delta^{\frac{dt}{t+d}-1}$, we have $k^{\frac{d}{t}+1}\delta^{\frac{d}{t}-d+1} > \delta^{d-\frac{t+d}{d} + \frac{t}{d} - d +1} = 1$. 
Therefore combining the Chung--Erd\H{o}s inequality~\eqref{e:chungerdosspecific} with Lemmas~\ref{nominator} and~\ref{denominator} gives 
\begin{equation*}
 \frac{1}{\delta^{d-1}} \cdot \mathcal{L}_{d-1} \left(\bigcup\limits_{n = l_t(k+1,\delta) + 1}^{l_t(k,\delta)} E^{(d-1)}(\delta,n) \right) \gtrsim \frac{1}{\delta^{d-1}} \cdot \frac{1}{k^{\frac{d}{t}+1}\delta^{\frac{d}{t}-d+1}} \cdot \left( \frac{1}{\log \log \delta^{-\frac{d}{t+d}}} \right)^{2d}.\qedhere
\end{equation*}
\end{proof}

We are now ready to give the proof of the lower bound for box dimension. 

\begin{proposition}\label{l:box}
If $0 < t < d/(d-1)$, then 
\[ \lbd G_{t,d} \geq \frac{d^2}{d+t}. \]
\end{proposition}

\begin{proof}
The idea is to bound the covering number (up to constants) from below by the sum of the covering numbers of an appropriate selection of layers of equal height, each of which can in turn be bounded using Lemma~\ref{CE-final}. 
For all sufficiently small $\varepsilon > 0$ satisfying
\[
0 < \varepsilon < \frac{1}{16} \cdot \min{ \left\{ \left( \frac{dt}{t+d} - \frac{t}{t+1}\right), \left(1- \frac{dt}{t+d}\right) \right\} }
\]
and sufficiently small $\delta > 0$, 
\[
 \bigcup_{k = \lfloor \delta^{\frac{dt}{t+d}-1-\varepsilon} \rfloor}^{\lfloor \delta^{\frac{t}{t+1}-1+\varepsilon} \rfloor} f_{t,d}(L_{t,d}(\delta,k)) \subseteq G_{t,d}.
\] 
It follows that 
\begin{align*}
N_{\delta}(G_{t,d}) &\gtrsim  N_\delta \left( \bigcup_{k = \lfloor \delta^{\frac{dt}{t+d}-1-\varepsilon} \rfloor}^{\lfloor \delta^{\frac{t}{t+1}-1+\varepsilon} \rfloor}f_{t,d}(L_{t,d}(\delta,k)) \right) \\
&\gtrsim \sum\limits_{k = \lfloor \delta^{\frac{dt}{t+d}-1-\varepsilon} \rfloor}^{\lfloor \delta^{\frac{t}{t+1}-1+\varepsilon} \rfloor} N_\delta \left( f_{t,d}(L_{t,d}(\delta,k)) \right) \\
&\gtrsim \sum\limits_{k = \lfloor \delta^{\frac{dt}{t+d}-1-\varepsilon} \rfloor}^{\lfloor \delta^{\frac{t}{t+1}-1+\varepsilon} \rfloor} \frac{1}{k^{\frac{d}{t}+1}\delta^{\frac{d}{t}}} \cdot \left( \frac{1}{\log \log \delta^{-\frac{d}{t+d}}}\right)^{2d} &\text{by Lemma~\ref{CE-final}} \\
& \gtrsim \left( \frac{1}{\log \log \delta^{-\frac{d}{t+d}}} \right)^{2d} \cdot \delta^{-\left(\frac{dt}{d+t} - 1 - 2\varepsilon \right) \cdot \frac{d}{t} + \frac{d}{t}}.
\end{align*}
But there exists $\delta_0 > 0$ depending on $\epsilon,t,d$ such that for all $\delta \in (0,\delta_0)$, $\log \log\delta^{-\frac{d}{t+d}} \leq \delta^{-\epsilon}$. Therefore for all sufficiently small $\delta$, 
\[ 
N_{\delta}(G_{t,d}) \gtrsim \delta^{-\left( \frac{d^2}{t+d} - 2d(1+t^{-1})\epsilon \right)}.
\]
By definition~\eqref{e:boxdimdef}, 
\[
\lbd G_{t,d} \geq\frac{d^2}{t+d} - 2d(1+t^{-1})\epsilon,
\]
 and letting $\epsilon \to 0^+$ completes the proof. 
\end{proof}

\section{Proof of dimension results}\label{s:proof2}
Next, we use Proposition~\ref{l:box} to prove the lower bound for the intermediate dimensions. 
\begin{lemma}\label{l:lower}
If $0 < t < d/(d-1)$ and $(d-1)t/d \leq \theta < 1$ then 
\[ \lid G_{t,d} \geq \frac{d^2 \theta}{d\theta + t}. \]
\end{lemma}
\begin{proof}
The idea is to transform a cover for $[\delta^{1/\theta},\delta]$ into a cover with balls of fixed size by `fattening' the smallest sets and breaking up the largest sets, and then use the formula for the box dimension. Fix $\theta \in ((d-1)t/d,1)$ and assume for contradiction that $\lid G_t < d^2 \theta /(d\theta + t)$. Then there exists $s \in (1, d^2 \theta /(d\theta + t))$ with $s < d^2 \theta /(d\theta + t)$ and a sequence of $\delta \to 0^+$ for which there exists a cover $\{U_i\}_{i \in I}$ of $G_t$ satisfying $\delta^{1/\theta} \leq |U_i| \leq \delta$ for all $i$, and $\sum_i |U_i|^s \leq 1$. Let $\beta \coloneqq (d+t)/(d\theta + t)$. We write $I = I_1 \cup I_2$ where 
\[ I_1 \coloneqq \{ \, i \in I : |U_i| \leq \delta^\beta \, \}; \qquad I_2 \coloneqq \{ \, i \in I : |U_i| > \delta^\beta \, \}. \] 
If $i \in I_1$ then we fix a ball $B_i$ of diameter $\delta^\beta$ such that $U_i \subseteq B_i$. 
For each $j \in I_2$, let $\gamma_j \in [1,\beta)$ be such that $|U_j| = \delta^{\gamma_j}$, so by a simple volume argument $U_j$ can be covered by a collection $S_j$ of balls of size $\delta^\beta$ with $\# S_j \lesssim \delta^{-d(\beta - \gamma_j)}$. 
Then 
\begin{equation}\label{e:cover}
 \{B_i\}_{i \in I_1} \cup \bigcup_{j \in I_2} S_j
 \end{equation}
is a cover of $G_t$. 
We will now bound the number of balls in this cover. 
First note that $\# I_1 \leq \# I \leq \delta^{-s/\theta}$. 
Also, 
\[\delta^{s + 2\beta - 2} \sum_{j \in I_2} \# S_j  
\lesssim \sum_{j \in I_2} \delta^{-d(\beta - \gamma_j) + s + d\beta - d}
\leq \sum_{j \in I_2} \delta^{\gamma_j s} = \sum_{j \in I_2} |U_j|^s \leq 1, \]
where the middle inequality holds since each $\gamma_j \geq 1$. 
Therefore the number of balls in the cover \eqref{e:cover} can be bounded above by 
\[ \# I_1 + \sum_{j \in I_2} \# S_j \lesssim \delta^{-s/\theta} + \delta^{-(s + d\beta - d)} \lesssim (\delta^{\beta})^{-\max\{s/(\theta \beta),(s + d\beta - d)/\beta\}}. \]
 But $\max\{s/(\theta \beta),(s + d\beta - d)/\beta\} < d^2/(t+d)$, contradicting Proposition~\ref{l:box}. 
\end{proof}

A shorter, but less direct, way to prove Lemma~\ref{l:lower} is to combine Proposition~\ref{l:box} with the bound~\eqref{e:generallower} with $\asd F$ replaced by $d$. 

Next, we give a direct proof of the upper bound for the intermediate dimensions. 

\begin{lemma}\label{l:upper}
If $0 < t < d/(d-1)$ and $(d-1)t/d \leq \theta \leq 1$ then  
\[ \uid F_{t,d} \leq \frac{d^2 \theta}{d\theta + t}. \] 
\end{lemma} 
\begin{proof}
The idea is to fix a small number $\delta > 0$ and separate $F_t$ into two parts: 
\begin{align*} 
F_{t,d}^{(1)} &= F_{t,d} \cap ([0,1]^{d-1} \times [0,\delta^{dt/(d\theta + t)}]), \\*
F_{t,d}^{(2)} &= F_{t,d} \cap ([0,1]^{d-1} \times (\delta^{dt/(d\theta + t)},1]).
\end{align*}
Covering $[0,1]^{d-1} \times [0,\delta^{dt/(d\theta + t)}]$ with balls of size $\delta$ gives   
\[ N_\delta (F_{t,d}^{(1)}) \lesssim \delta^{-(d-1)} \delta^{dt/(d\theta + t) - 1} = \delta^{-d^2\theta/(d\theta+t)}. \]
It follows from a simple cardinality estimate that 
\[ \# F_{t,d}^{(2)} \lesssim (\delta^{-d/(d\theta + t)})^d = \delta^{-d^2/(d\theta + t)}. \]

We can cover each point in $F_{t,d}^{(2)}$ with a ball of size $\delta^{1/\theta}$ and the result now follows from the estimate  
\[ N_\delta (F_{t,d}^{(1)}) \cdot \delta^{d^2\theta/(d\theta + t)} + \# F_{t,d}^{(2)} \cdot (\delta^{1/\theta})^{d^2\theta/(d\theta + t)} \lesssim 1. \qedhere\]
\end{proof}

An alternative way to prove Lemma~\ref{l:upper} is to combine the simple special case 
\[ {\dim_{(d-1)t/d} F_{t,d} \leq d-1} \] with~\eqref{e:generalbound} with $\asd F$ replaced by $d$. 

We are finally ready to complete the proof of our dimension results Theorem~\ref{t:int} (from which Theorem~\ref{t:box} follows by setting $\theta=1$) and Theorem~\ref{t:assouad}. 
\begin{proof}[Proof of Theorem~\ref{t:int}]
Since the intermediate dimensions are clearly non-decreasing in $\theta$, it follows from Lemma~\ref{l:lower} that $\uid F_{t,d} \leq {\overline{\dim}}_{(d-1)t/d} F_{t,d} \leq d-1$ for all $\theta \in [0,(d-1)t/d]$. Moreover, $\lid G_{t,d} \geq \dim_{\mathrm H} G_{t,d} = d-1$ for all $\theta \in [0,1]$. Since $G_{t,d} \subset F_{t,d}$, we have $\lid G_{t,d} \leq \lid F_{t,d} \leq \uid F_{t,d}$ and $\lid G_{t,d} \leq \uid G_{t,d} \leq \uid F_{t,d}$, so Theorem~\ref{t:int} now follows from Lemmas~\ref{l:lower} and~\ref{l:upper}. 
\end{proof}

\begin{proof}[Proof of Theorem~\ref{t:assouad}]
The $t<d/(d-1)$ case follows by combining Theorem~\ref{t:int} with the general bound~\eqref{e:generallower}. Therefore it suffices to prove that $\asd F_{t,d} \leq d-1$ when $t \geq d/(d-1)$; we use similar methods to the $d=2$ case in~\cite{Chen2022tangent}. Indeed, for each $0 < r < R < 1$ there are unique $m,n \in \N$ with $m \geq n$ such that 
\[ 
\frac{1}{n} \leq R < \frac{1}{n-1}, \qquad \frac{1}{m} \leq r < \frac{1}{m-1}.
\]
Given $x = (x_1,\dotsc,x_d) \in \Rd$ with each $x_i \geq 0$, define 
\[
C(x,R) \coloneqq F_{t,d} \cap \prod_{i=1}^d [x_i, x_i + R].
\] 
Define 
\begin{align*}
C_1 &\coloneqq C\left( x,\frac{1}{n-1}\right) \cap \left(\mathbb{R}^{d-1} \times \left( \frac{1}{m},\infty\right)\right) \\*
 C_2 &\coloneqq C\left( x,\frac{1}{n-1}\right) \cap \left(\mathbb{R}^{d-1} \times \left[ 0,\frac{1}{m}\right]\right)
\end{align*}
to be regions separated by whether the last coordinate is larger or smaller than $1/m$. 
A simple counting argument shows that the number of points in $F_{t,d}$ whose last coordinate is greater than $1/m$ is $\lesssim m^{d/t}$, and that 
\[
N_{1/m}(C_1) \leq \# C_1 \lesssim \left(\frac{1}{n}\right)^{d-1} m^{d/t} \leq \frac{m^{d-1}}{n^{d-1}},
\]
where we used that $t \geq d/(d-1)$ in the last step. 
Moreover, covering $C_2$ using a $1/m$-mesh of cubes gives 
\[ 
N_{1/m}(C_2) \lesssim \left(\frac{1/(n-1)}{1/m}\right)^{d-1} \lesssim \left( \frac{m}{n}\right)^{d-1}.
\]
Therefore
\[ 
N_r(C(x,R)) \leq N_{1/m}(C(x,1/n)) \leq N_{1/m}(C_1) + N_{1/m}(C_2) \lesssim \left( \frac{m}{n}\right)^{d-1} \lesssim \left( \frac{R}{r} \right)^{d-1},
\]
so $\asd F_{t,d} \leq d-1$, as required. 
\end{proof}

\section*{Acknowledgements}
We thank Jonathan Fraser for pointing out that the bounds~\eqref{e:generalbound} and~\eqref{e:generallower} can be used to calculate the intermediate dimensions of the graph of the popcorn function, and for other helpful comments. We thank Kenneth Falconer and Alex Rutar for comments on a draft version of this article. 
AB was financially supported by a Leverhulme Trust Research Project Grant (RPG-2019-034). HC was supported by NSFC (No. 11871227) and Shenzhen Science and Technology Program (Grant No. RCBS20210706092219049).

\section*{References}

\printbibliography[heading=none]

\bigskip
\footnotesize
{\parindent0pt
\textsc{Amlan Banaji\\ School of Mathematics and Statistics \\ University of St Andrews \\ St Andrews, KY16 9SS, UK}\par\nopagebreak
\textit{Email:} \texttt{afb8@st-andrews.ac.uk}
}

\vspace{5 pt}

{\parindent0pt
\textsc{Haipeng Chen\\ College of Big Data and Internet \\ Shenzhen Technology University \\ Shenzhen, 518118, China}\par\nopagebreak
\textit{Email:} \texttt{hpchen0703@foxmail.com}
}

\end{document}